\documentclass[11pt]{amsart}
\usepackage[T1]{fontenc}
\usepackage[utf8]{inputenc} 
\usepackage[
  backend=biber,
  style=numeric,
  sorting=none,
  giveninits=false,
  doi=true,
  isbn=true,
  url=true,
  eprint=true
]{biblatex}

\usepackage{lmodern}
\usepackage{microtype}

\usepackage{amsmath, amssymb, amsfonts, amsthm}
\usepackage{mathtools}
\usepackage{enumitem}
\usepackage{hyperref}
\usepackage{cleveref}
\usepackage{tikz-cd}
\usepackage{geometry}
\usepackage{thmtools}
\usepackage{float}

\hypersetup{
  colorlinks=true,
  linkcolor=blue,
  citecolor=blue,
  urlcolor=blue,
  pdftitle={Perverse Schober Models for Conifold Degenerations},
  pdfauthor={Abdul Rahman}
}
\numberwithin{equation}{section}
\theoremstyle{plain}
\newtheorem{theorem}{Theorem}[section]
\newtheorem{proposition}[theorem]{Proposition}
\newtheorem{lemma}[theorem]{Lemma}

\theoremstyle{definition}
\newtheorem{definition}[theorem]{Definition}

\theoremstyle{remark}
\newtheorem{remark}{Remark}[section]

\newcommand{\RHom}{\mathbf{R}\!\operatorname{Hom}}

\newcommand{\Perv}{\mathrm{Perv}}

\newcommand{\Ext}{\mathrm{Ext}}
\newcommand{\IC}{\mathrm{IC}}
\newcommand{\can}{\mathrm{can}}
\newcommand{\var}{\mathrm{var}}
\newcommand{\coker}{\mathrm{coker}}

\DeclareMathOperator{\Cone}{Cone}

\title[Perverse Schober Structures for Conifold Degenerations]{Perverse Schober Structures for Conifold Degenerations}

\author{Abdul Rahman }
\thanks{Email: arahman@alum.howard.edu}

\begin{document}

\begin{abstract}
We study a one-parameter degeneration of Calabi–Yau threefolds whose central fiber contains a single ordinary double point. Using the nearby and vanishing cycle formalism, we construct a canonical perverse object on the singular fiber from the variation morphism between vanishing and nearby cycles. We show that this object restricts to the constant perverse sheaf on the smooth locus and differs from the intersection complex by a single rank-one contribution supported at the node. Thus the object isolates the vanishing-cycle contribution associated with the conifold degeneration in a canonical sheaf-theoretic form. We also explain how this construction aligns with the rank-one Picard–Lefschetz phenomenon that appears categorically through spherical monodromy, making it a natural comparison object for the decategorified effect of spherical twists in the ordinary double point case.
\end{abstract}

\maketitle

\tableofcontents

\section{Introduction}

Degenerations of Calabi--Yau threefolds occupy a central position at the interface of topology, Hodge theory, and categorical structures arising in mirror symmetry. Among these, conifold degenerations provide the simplest and most structurally transparent model of topology change. An ordinary double point (ODP) singularity produces a vanishing $3$--sphere in the smoothing and modifies the middle-dimensional topology in a controlled and computable fashion
\cite{MilnorSingularPoints,DimcaSheaves,Clemens}. Globally, such transitions interpolate between Calabi--Yau manifolds with distinct Hodge numbers and play a foundational role in Reid's web of Calabi--Yau threefolds \cite{Clemens,StromingerConifold}.

Conifold degenerations also play a fundamental role in string theory and mirror symmetry. In the seminal work of Strominger
\cite{StromingerConifold}, conifold transitions were interpreted
physically as phase transitions in which certain BPS states become
massless when a three-cycle in the Calabi--Yau manifold collapses.
Geometrically, the transition is governed by the vanishing of a
Lagrangian $S^3$ and the associated monodromy acting on the middle
cohomology of the smooth fiber.

The classical description of this monodromy is given by the
Picard--Lefschetz formula. Let
\[
\pi : \mathcal{X} \to \Delta
\]
be a one--parameter degeneration over a small disk
\[
\Delta := \{\, t \in \mathbb{C} \mid |t| < \varepsilon \,\},
\]
and assume that $\mathcal X$ has pure complex dimension $4$, that the
fibers
\[
X_t := \pi^{-1}(t)
\]
are smooth compact Calabi--Yau threefolds for $t\neq0$, and that the
central fiber
\[
X_0 := \pi^{-1}(0)
\]
contains a single isolated ordinary double point $p\in X_0$.  Associated
with this singularity is a vanishing three--sphere in the Milnor fiber
whose homology class we denote
\[
\delta \in H_3(X_t,\mathbb Z).
\]
Parallel transport around a small loop encircling $t=0$ induces the
monodromy operator
\[
T: H^3(X_t,\mathbb Z)\longrightarrow H^3(X_t,\mathbb Z),
\]
which acts by the Picard--Lefschetz reflection
\[
T(\alpha)
=
\alpha + (\alpha\cdot\delta)\,\delta .
\]
This transformation is a rank--one unipotent operator determined by the
vanishing sphere and governs the topology of conifold transitions
\cite{MilnorSingularPoints,Clemens}.

From the perspective of homological mirror symmetry,
Picard--Lefschetz transformations admit a categorical realization. In
the derived category of coherent sheaves on the smooth fiber, the
vanishing cycle is mirrored by a spherical object whose associated
spherical twist induces a reflection on additive invariants such as the
Grothendieck group \cite{SeidelThomas}. More generally, Kapranov and
Schechtman introduced the notion of a \emph{perverse schober} as a
categorification of a perverse sheaf, providing a framework in which
categorical monodromy over a disk is governed by spherical functors and
their twists \cite{KapranovSchechtman}. In this setting the resulting perverse object reflects the same rank-one Picard–Lefschetz transformation that appears categorically through spherical twists. 

On the topological and Hodge--theoretic side, degenerations are governed
by the nearby and vanishing cycle formalism. The vanishing cycle
contributes a rank--one variation in middle cohomology, and the
resulting monodromy operator is unipotent of index two. At the level of
Hodge theory this produces a limiting mixed Hodge structure whose weight
filtration is explicitly computable
\cite{Schmid,SaitoMHM}. The ordinary double point therefore provides an
ideal model for understanding how degeneration data is organized
simultaneously at the topological, Hodge--theoretic, and categorical
levels.

Sheaf-theoretically, singular spaces admit canonical extension objects via intersection homology and the intersection complex. Let
\[
D^b_c(X_0)
\]
denote the bounded derived category of constructible
$\mathbb Q$--sheaf complexes on $X_0$. This category carries Verdier
duality
\[
\mathbb D : D^b_c(X_0)\longrightarrow D^b_c(X_0),
\qquad
\mathbb D^2 \cong \mathrm{Id},
\]
compatible with the perverse $t$--structure \cite{BBD,KS}. The
intersection complex $\IC_{X_0}$ provides the canonical
perverse extension of the constant sheaf across the singular
point.

A closely related but distinct sheaf-theoretic construction was
developed by Banagl, Budur, and Maxim, who introduced the
intersection-space complex as a perverse-sheaf model for
intersection-space cohomology and established mixed-Hodge and
self-duality properties under suitable hypotheses
\cite{BanaglBudurMaxim}. The present paper has a different emphasis.
Rather than recovering intersection-space cohomology, we isolate a
canonical perverse object in the ordinary double point setting directly
from the nearby/vanishing-cycle formalism and interpret it as a natural
sheaf-theoretic comparison object for rank--one categorical monodromy.

These developments suggest a natural structural question:
\begin{quote}
\textit{Given a geometric degeneration governed by vanishing cycles, is there a canonical perverse object on the singular fiber whose invariants reflect the same rank--one monodromy phenomenon that appears categorically
through spherical twists?}
\end{quote}

The goal of this paper is to identify such an object in the simplest
nontrivial case of an ordinary double point degeneration of
Calabi--Yau threefolds. Using the nearby and vanishing cycle
distinguished triangle, we construct a canonical perverse object
\[
\mathcal P \in \Perv(X_0),
\]
where \(\Perv(X_0)\) denotes the heart of the perverse \(t\)-structure
on \(D^b_c(X_0)\) \cite{BBD}. We show that \(\mathcal P\) restricts to
the constant perverse sheaf on the smooth locus and differs from the
intersection complex by a single rank--one contribution supported at the
node. Equivalently, \(\mathcal P\) fits into a short exact sequence
\[
0 \longrightarrow \IC_{X_0}
\longrightarrow \mathcal P
\longrightarrow i_*\mathbb Q_{\{p\}}
\longrightarrow 0 .
\]

We then explain how this object naturally reflects the same rank--one
phenomenon that appears categorically through spherical monodromy. In
this sense, \(\mathcal P\) reflects the same rank-one Picard–Lefschetz phenomenon that appears categorically through spherical twists.

We emphasize that this work does not introduce a new cohomology theory
and does not replace intersection homology or intersection-space
constructions \cite{BanaglIS,BanaglMaximDeformation}. Rather, it
identifies the perverse object singled out functorially by the
nearby/vanishing-cycle formalism and positions it as the natural
sheaf-theoretic object attached to the rank--one vanishing contribution
of the conifold degeneration.

Recent constructions of perverse schobers in mirror symmetry reinforce
this viewpoint. For example, Koseki and Ouchi construct schobers on the
Riemann sphere arising from Calabi--Yau hypersurfaces using Orlov
equivalences and graded matrix factorizations to categorify intersection
complexes associated with mirror monodromy representations
\cite{KosekiOuchi}. Their work provides concrete evidence that
degeneration phenomena in Calabi--Yau geometry admit categorical lifts
whose decategorification recovers classical invariants.

The ordinary double point provides a setting in which topology, Hodge
theory, and categorical monodromy can be compared explicitly. In this
case the vanishing cycle is rank one, the monodromy is unipotent of
index two, and the corresponding spherical twist is generated by a
single spherical object. This allows the classical and categorical
pictures to be related in a particularly transparent way.

The main contributions of this paper are:
\begin{enumerate}[label=(\arabic*)]
\item We construct a canonical perverse object
\(\mathcal P\in\Perv(X_0)\) attached functorially to a conifold
degeneration via nearby and vanishing cycles.
\item We show that \(\mathcal P\) restricts to the constant perverse
sheaf on the smooth locus and differs from the intersection complex by a single rank--one point-supported summand at the node.
\item We explain how \(\mathcal P\) reflects the same rank-one Picard–Lefschetz phenomenon that appears categorically through spherical twists.
\end{enumerate}

Finally, we indicate how the same framework extends to degenerations
with multiple nodes and more general singular strata, and we discuss
possible refinements in the direction of mixed Hodge modules and
Hodge-enhanced categorical structures.

\subsection{Main results}
\label{sec:intro-results}

Let \(\pi:\mathcal X\to\Delta\) be a one--parameter degeneration of
Calabi--Yau threefolds whose central fiber \(X_0\) has a single
ordinary double point. Let \(\psi_\pi\) and \(\phi_\pi\) denote the
nearby and vanishing cycle functors associated with \(\pi\). Our first main result identifies a canonical perverse object encoding the local vanishing contribution of the degeneration.

\begin{theorem}[Canonical perverse model]\label{thm:main-perverse}
Let $\pi:\mathcal X\to\Delta$ be a one-parameter degeneration of
Calabi--Yau threefolds whose central fiber $X_0$ has a single
ordinary double point $p$. Let
\[
F:=\mathbb Q_{\mathcal X}[3].
\]
Assume that $\var_F:\phi_\pi(F)\to\psi_\pi(F)$ is a monomorphism in
$\Perv(X_0)$. Then the object
\[
\mathcal P:=\Cone(\var_F)[-1]
\]
lies in $\Perv(X_0)$ and satisfies:

\begin{enumerate}[label=(\roman*)]
\item $j^*\mathcal P\cong \mathbb Q_U[3]$ on the smooth locus $U:=X_0\setminus\{p\}$;
\item there is a short exact sequence
\[
0\to \IC_{X_0}\to \mathcal P\to i_*\mathbb Q_{\{p\}}\to 0;
\]
\item $\dim_{\mathbb Q}{}^pH^0(i^*\mathcal P)=1$.
\end{enumerate}
\end{theorem}

\begin{remark}
In the ordinary double point case, the object \(\mathcal P\) is defined by
\[
\mathcal P := \Cone(\var_F)[-1],
\qquad
F:=\mathbb Q_{\mathcal X}[3],
\]
where \(\var_F:\phi_\pi(F)\to\psi_\pi(F)\) is the variation morphism. The variation morphism arises from the distinguished triangle relating nearby and vanishing cycles in $D^b_c(X_0)$ \cite{DimcaSheaves}.
The shifted cone lies in the abelian category \(\Perv(X_0)\) because
\(\var_F\) is a monomorphism in the perverse heart. In particular, this isolates precisely the rank-one vanishing contribution associated with the ordinary double point in a canonical way and therefore provides the natural perverse object controlling the Picard–Lefschetz sector of the degeneration.
\end{remark}
Our second main result concerns the categorical interpretation of this
construction.

\begin{proposition}[Categorical comparison, conditional form]
\label{prop:main-decategorification}
Assume the categorical monodromy of the degeneration is governed by a
spherical object $S\in D^b(X_t)$, and assume there is an additive realization
\[
\rho:K_0(D^b(X_t))\to H^3(X_t,\mathbb Q)
\]
sending $[S]$ to the vanishing cycle and intertwining the spherical twist with the Picard--Lefschetz monodromy operator. Then the perverse sheaf $\mathcal P$ provides a natural sheaf-theoretic comparison object for the corresponding rank-one decategorified monodromy phenomenon.
\end{proposition}

Taken together, these results exhibit \(\mathcal P\) as the canonical
perverse object naturally attached to the ordinary double point
degeneration and explain why it is the correct sheaf-theoretic shadow of the corresponding rank--one categorical monodromy.

\section{Related work and context}
\label{sec:related}

The present paper lies at the intersection of three well-developed
strands of the literature:
\begin{enumerate}[label=(\roman*)]
\item sheaf-theoretic vanishing-cycle formalisms for degenerations,
\item perverse models motivated by conifold
transitions in complex dimension three,
\item categorical Picard--Lefschetz theory and its relation to
spherical twists and perverse schobers.
\end{enumerate}
We summarize below the points most relevant to the construction
introduced here.

\subsection{Perverse models in conifold settings}

In earlier work, the author constructed a perverse
sheaf in the conifold context using MacPherson--Vilonen techniques
\cite{RahmanATMP}.  That construction provides an explicit
sheaf-theoretic mechanism for enforcing duality constraints in the
presence of isolated singularities and motivates the present focus on
canonical perverse objects attached functorially to degenerations.

Related constructions appear in the intersection-space program of
Banagl, Budur, and Maxim \cite{BanaglBudurMaxim}.  For complex
projective hypersurfaces with isolated singularities, they show that
intersection-space cohomology can be realized as the hypercohomology of
a perverse sheaf and establish self-duality properties under suitable
hypotheses.  More broadly, the intersection-space program initiated by
Banagl \cite{BanaglIS} provides a homotopy-theoretic framework for
constructing duality-satisfying invariants of singular spaces via
spatial homology truncation.

A particularly relevant feature for conifold and hypersurface
degenerations is that deformation behavior in middle degree is governed
by Milnor-fiber monodromy.  Banagl and Maxim show that for complex
projective hypersurfaces with an isolated singularity,
intersection-space cohomology is stable under smooth deformation in all
degrees except possibly the middle, where the correction is governed by
local monodromy \cite{BanaglMaximDeformation}.  This provides further
evidence that vanishing-cycle and monodromy data control the
middle-dimensional correction terms in duality-compatible theories.

We emphasize, however, that the perverse object constructed in the
present paper is formally distinct from the intersection-space complex
of \cite{BanaglBudurMaxim}.  In the present work the object
\[
\mathcal P
=
\Cone\!\left(
\phi_\pi(\mathbb Q_{\mathcal X}[3])
\longrightarrow
\psi_\pi(\mathbb Q_{\mathcal X}[3])
\right)[-1]
\]
is defined directly and functorially from the
nearby/vanishing-cycle triangle.  The emphasis of the present paper is different: rather than realizing a specific cohomology theory, we isolate the perverse object that arises canonically from the nearby/vanishing cycle triangle and
analyze its extension structure in the ordinary double point case.

\subsection{Categorical Picard--Lefschetz theory and spherical monodromy}

On the categorical side, spherical twists and spherical functors provide
categorical analogues of Picard--Lefschetz monodromy.
Seidel and Thomas introduced spherical twists in derived categories and
established the braid-type structures generated by such twists
\cite{SeidelThomas}.  Anno and Logvinenko developed a framework for
spherical functors and their associated twist autoequivalences
\cite{AnnoLogvinenko}.  More broadly, Katzarkov, Pandit, and Spaide
formulate a categorical Picard--Lefschetz perspective in which monodromy
phenomena are organized by spherical functors and Calabi--Yau
structures \cite{KatzarkovPanditSpaide}.

More recently, Christ, Dyckerhoff, and Walde introduced categorical
complexes and higher-dimensional perverse schobers, providing a setting
in which Picard--Lefschetz phenomena appear at the level of stable
$\infty$-categories \cite{CDW}.

These developments motivate the basic question addressed here: how the
classical vanishing-cycle formalism of a geometric degeneration is
reflected in a natural sheaf-theoretic object compatible with the
decategorified effect of spherical monodromy.

\subsection{Perverse schobers}

Kapranov and Schechtman proposed perverse schobers as categorical
analogues of perverse sheaves, replacing vector spaces by triangulated
or stable categories and singular gluing data by exact functors
\cite{KapranovSchechtman}.  In the local disk model with one singular
point, the monodromy is governed by a spherical functor, providing a
categorical analogue of local Picard--Lefschetz behavior.

The present work does not construct a full schober.  Rather, it
identifies a canonical perverse sheaf on the singular fiber that may be
viewed as the natural sheaf-theoretic shadow of the same rank--one
monodromy phenomenon.

\subsection{Position of the present work}

We work in the model case of a one--parameter Calabi--Yau threefold
degeneration with a single ordinary double point.  Our contribution is
to isolate a canonical perverse object on the singular fiber,
constructed functorially from nearby and vanishing cycle data, whose
local and extension-theoretic structure can be described explicitly.
We then interpret this object as a natural sheaf-theoretic comparison
object for the corresponding rank--one categorical monodromy.

\begin{remark}[Canonical vs.\ rank-prescribed perverse models]
\label{rem:can-vs-rank}
There are at least two natural perverse objects that arise in conifold
settings.  The object constructed in the present paper,
\[
\mathcal P
:=
\Cone\!\big(
\phi_\pi(\mathbb Q_{\mathcal X}[3])
\to
\psi_\pi(\mathbb Q_{\mathcal X}[3])
\big)[-1],
\]
is functorially singled out by the nearby/vanishing-cycle distinguished
triangle.  By contrast, the perverse models
appearing in \cite{RahmanATMP} are obtained by imposing additional
local constraints via explicit zig--zag / MacPherson--Vilonen data
\cite{MacPhersonVilonen1986}.  Consequently, \(\mathcal P\) need not
coincide with the rank-prescribed perverse sheaf of
\cite{RahmanATMP}, even in the ordinary double point case.
\end{remark}

\begin{remark}[Relation with earlier constructions]

The use of perverse sheaves to model the cohomology of a
Calabi--Yau degeneration with an isolated node goes back to the
construction proposed in \cite{HubschRahman2002}.  That work was
motivated by the linear algebra description of perverse sheaves on a
space with a single singular stratum developed by MacPherson and
Vilonen \cite{MacPhersonVilonen1986}.  Banagl subsequently observed
that this perverse sheaf provides a natural candidate for describing
string-theoretic cohomology of conifold transitions and cited it as
the first ansatz in this direction.

The present paper revisits this construction from the perspective of
nearby and vanishing cycles and clarifies the extension structure
underlying the resulting perverse object.
\end{remark}

To the best of the author's knowledge, while the ingredients used here
are all well represented in the literature --- perverse models in conifold settings
\cite{RahmanATMP,BanaglBudurMaxim}, categorical Picard--Lefschetz
theory via spherical twists
\cite{SeidelThomas,AnnoLogvinenko,KatzarkovPanditSpaide,CDW}, and
schober-theoretic interpretations of perverse phenomena
\cite{KapranovSchechtman} --- the specific construction presented
here, namely a perverse object obtained directly from the
nearby/vanishing-cycle triangle and interpreted as a natural
sheaf-theoretic comparison object for rank--one spherical monodromy in
the ordinary double point setting, does not appear to be recorded in
the literature in this form.

\section{Conifold degenerations, nearby and vanishing cycles}
\label{sec:conifold}

In this section we describe the local geometry of a conifold
degeneration and reformulate the associated Picard--Lefschetz
phenomena in terms of nearby and vanishing cycle functors.
These functors provide a sheaf--theoretic framework for encoding
the topology of the degeneration.

Throughout we work in the bounded derived category of
constructible $\mathbb Q$--sheaves $D^b_c(\mathcal X)$.

\subsection{Geometric setup}

Let
\[
\pi:\mathcal X\to\Delta
\]
be a one--parameter smoothing of a Calabi--Yau threefold with a
single ordinary double point.

The central fiber
\[
X_0=\pi^{-1}(0)
\]
is a singular Calabi--Yau threefold, while for $t\neq0$ the fibers
\[
X_t=\pi^{-1}(t)
\]
are smooth. Locally near the singular point $p\in X_0$ the degeneration is
analytically equivalent to
\[
x_1^2+x_2^2+x_3^2+x_4^2=t
\subset\mathbb C^4\times\Delta .
\]
Hence, the total space $\mathcal X$ has complex dimension $4$,
each fiber $X_t$ has complex dimension $3$.

\begin{remark}[Dimension of the total space]
Although the Calabi--Yau variety $X_0$ has complex dimension $3$,
the total space of the degeneration has dimension $4$.
This situation is standard in the theory of degenerations and
limiting mixed Hodge structures
\cite{Schmid,SteenbrinkLimits,SaitoMHM}.
\end{remark}

\subsection{Milnor fiber and vanishing cohomology}

Let
\[
f(x)=x_1^2+x_2^2+x_3^2+x_4^2 .
\]
For $0<|\delta|\ll1$ the Milnor fiber is
\[
F=f^{-1}(\delta)\cap B_\varepsilon(0)
\]
for a sufficiently small ball $B_\varepsilon(0)$. For an ordinary double point the Milnor fiber satisfies \cite{MilnorSingularPoints,DimcaSheaves}
\[
F\simeq S^3.
\]

\begin{lemma}[Vanishing cohomology for an ODP]
\label{lem:vanishing-rank-one}

Let $X_t$ be a nearby smooth fiber.
Then

\[
H^k_{\mathrm{van}}(X_t,\mathbb Q)=0 \quad (k\neq3),
\]

and

\[
H^3_{\mathrm{van}}(X_t,\mathbb Q)\cong\mathbb Q .
\]

\end{lemma}

\begin{proof}

The vanishing cohomology identifies with the reduced cohomology
of the Milnor fiber
\cite{MilnorSingularPoints,DimcaSheaves}:

\[
H^k_{\mathrm{van}}(X_t,\mathbb Q)
\cong
\widetilde H^k(F,\mathbb Q).
\]

Since $F\simeq S^3$ the result follows.

\end{proof}

\subsection{Nearby and vanishing cycles}

Let
\[
i:X_0\hookrightarrow\mathcal X,
\qquad
j:\mathcal X^\ast:=\pi^{-1}(\Delta^\ast)\hookrightarrow\mathcal X .
\]
The nearby cycle functor is defined by
\[
\psi_\pi(\mathcal F):=i^*Rj_*j^*\mathcal F .
\]
The vanishing cycle functor is
\[
\phi_\pi(\mathcal F)
:=
\Cone\!\big(i^*\mathcal F\to\psi_\pi(\mathcal F)\big)[-1].
\]
These fit into the distinguished triangle
\begin{equation} \label{eqn:van-func-dist-triangle}
\begin{tikzpicture}[baseline=(current bounding box.center),>=stealth]
  \node (A) at (0,0) {$\phi_\pi(\mathcal F)$};
  \node (B) at (3.6,0) {$\psi_\pi(\mathcal F)$};
  \node (C) at (1.8,-1.5) {$i^*\mathcal F$};
  \draw[->] (A) -- (B);
  \draw[->] (B) -- (C);
  \draw[->,bend left=20] (C) to node[left] {$+1$} (A);
\end{tikzpicture}
\end{equation}
in $D^b_c(X_0)$ \cite{DimcaSheaves,BBD}. 

\subsection{Canonical morphisms}

The nearby and vanishing cycle functors are equipped with
natural morphisms

\[
\var_{\mathcal F}:\phi_\pi(\mathcal F)\to\psi_\pi(\mathcal F),
\]

\[
\can_{\mathcal F}:\psi_\pi(\mathcal F)\to i^*\mathcal F .
\]
These satisfy
\[
\can_{\mathcal F}\circ\var_{\mathcal F}
=
T-\mathrm{id},
\]
where $T$ is the monodromy operator on nearby cycles \cite{DimcaSheaves}.

\subsection{Vanishing cycles for an ordinary double point}

\begin{proposition}
\label{prop:odp-vanishing}

Let $F=\mathbb Q_{\mathcal X}[3]$. Then
\[
\phi_\pi(F)\cong i_*\mathbb Q_{\{p\}}
\]
in $\Perv(X_0)$.
\end{proposition}

\begin{proof}
The stalk cohomology of vanishing cycles is identified with the
reduced cohomology of the Milnor fiber \cite{DimcaSheaves}. For an ordinary double point, the Milnor fiber is homotopy equivalent to $S^3$, so after the shift $[3]$, there is a single nonzero perverse stalk at $p$.
Thus $\phi_\pi(F)$ is the skyscraper perverse sheaf $i_*\mathbb Q_{\{p\}}$.

\end{proof}

\subsection{Monodromy}

The nearby cycle complex carries a canonical monodromy
automorphism
\[
T:\psi_\pi(\mathcal F)\to\psi_\pi(\mathcal F)
\]
induced by analytic continuation around the punctured disk
\cite{DimcaSheaves}. For the conifold singularity the Picard--Lefschetz formula gives
\[
T(\alpha)=\alpha+\langle\alpha,\delta\rangle\,\delta .
\]
Hence $T$ is unipotent and
\[
(T-\mathrm{Id})^2=0.
\]

\section{Canonical perverse models}
\label{sec:canonical-model}

In this section we construct the canonical perverse object associated
with the conifold degeneration of Section~\ref{sec:conifold}.  The
construction uses only the nearby and vanishing cycle formalism
recalled there.  Throughout we write
\[
F:=\mathbb Q_{\mathcal X}[3].
\]

\subsection{The variation morphism and the definition of \texorpdfstring{$\mathcal P$}{P}}

By Section~\ref{sec:conifold}, the nearby and vanishing cycle functors
carry canonical morphisms
\[
\can_F:\psi_\pi(F)\longrightarrow i^*F,
\qquad
\var_F:\phi_\pi(F)\longrightarrow\psi_\pi(F),
\]
satisfying
\[
\can_F\circ \var_F = T-\mathrm{id}
\]
on nearby cycles \cite{DimcaSheaves}. For the ordinary double point, Proposition~\ref{prop:odp-vanishing} identifies
\[
\phi_\pi(F)\cong i_*\mathbb Q_{\{p\}}
\]
in $\Perv(X_0)$, and Lemma~\ref{lem:vanishing-rank-one} shows that the
vanishing cohomology is one-dimensional.

\begin{proposition}
\label{prop:var-mono}
The variation morphism
\[
\var_F:\phi_\pi(F)\longrightarrow \psi_\pi(F)
\]
is a monomorphism in the abelian category $\Perv(X_0)$.
\end{proposition}

\begin{proof}
By Proposition~\ref{prop:odp-vanishing}, the source
\[
\phi_\pi(F)\cong i_*\mathbb Q_{\{p\}}
\]
is a simple perverse sheaf, since perverse sheaves supported at a
point are equivalent to finite-dimensional $\mathbb Q$-vector spaces.
Thus it is enough to show that $\var_F$ is nonzero. By the identity
\[
\can_F\circ\var_F=T-\mathrm{id}
\]
and the Picard--Lefschetz formula of
Section~\ref{sec:conifold}, the endomorphism
\[
T-\mathrm{id}
\]
has rank one on the vanishing cohomology. In particular,
\[
T-\mathrm{id}\neq 0.
\]
Hence $\var_F\neq 0$, and therefore $\var_F$ is injective.
\end{proof}

\begin{definition}
\label{def:P}
The canonical perverse object associated with the degeneration
\(\pi\) is defined by
\[
\mathcal P:=\Cone(\var_F)[-1].
\]
\end{definition}

\begin{lemma}[Cones of monomorphisms in the heart]
\label{lem:cone-heart}
Let $\mathcal D$ be a triangulated category equipped with a
$t$--structure with heart $\mathcal A$.  If
\[
f:A\to B
\]
is a monomorphism in $\mathcal A$, then
\[
\Cone(f)[-1]\in \mathcal A
\]
and there is a short exact sequence
\[
0\to A\xrightarrow{\,f\,}B\to \Cone(f)[-1]\to 0
\]
in $\mathcal A$.
\end{lemma}

\begin{proof}
Let
\[
A \xrightarrow{f} B \to C \to A[1]
\]
be the distinguished triangle defining $C=\Cone(f)$.
Since $A,B\in\mathcal A$, the only nonzero cohomology objects of $C$
with respect to the given $t$-structure can occur in degrees $0$ and $1$. Because $f$ is a monomorphism in $\mathcal A$, we have
$H^0(C)\cong \coker(f)$ and $H^1(C)\cong \ker(f)=0$. Hence $C[-1]\in\mathcal A$, and the induced exact sequence in the heart is
\[
0\to A\xrightarrow{f}B\to C[-1]\to 0.
\]
\end{proof}

\begin{proposition}
\label{prop:P-perverse}
The object \(\mathcal P\) lies in \(\Perv(X_0)\).  More precisely,
there is a short exact sequence
\[
0\to i_*\mathbb Q_{\{p\}}
\xrightarrow{\ \var_F\ }
\psi_\pi(F)
\longrightarrow
\mathcal P
\to 0
\]
in \(\Perv(X_0)\).
\end{proposition}

\begin{proof}
By $t$-exactness of nearby and vanishing cycles for the perverse
$t$-structure, both $\phi_\pi(F)$ and $\psi_\pi(F)$ are perverse sheaves. By Proposition~\ref{prop:var-mono}, the morphism
\[
\var_F:\phi_\pi(F)\to\psi_\pi(F)
\]
is a monomorphism in $\Perv(X_0)$. Therefore Lemma~\ref{lem:cone-heart} applies and shows that
\[
\mathcal P:=\Cone(\var_F)[-1]
\]
lies in $\Perv(X_0)$, with short exact sequence
\[
0\to \phi_\pi(F)\xrightarrow{\var_F}\psi_\pi(F)\to\mathcal P\to 0.
\]
Using Proposition~\ref{prop:odp-vanishing}, this becomes
\[
0\to i_*\mathbb Q_{\{p\}}\xrightarrow{\var_F}\psi_\pi(F)\to\mathcal P\to 0.
\]
\end{proof}

\subsection{Restriction to the smooth locus}

Let
\[
U:=X_0\setminus\{p\},
\qquad
j:U\hookrightarrow X_0,
\qquad
i:\{p\}\hookrightarrow X_0.
\]

\begin{proposition}
\label{prop:P-restriction}
There is a canonical isomorphism
\[
j^*\mathcal P\cong \mathbb Q_U[3].
\]
\end{proposition}

\begin{proof}
Since $\phi_\pi(F)$ is supported at the singular point $p$, we have
\[
j^*\phi_\pi(F)=0.
\]
Applying $j^*$ to the short exact sequence of Proposition~\ref{prop:P-perverse} gives
\[
j^*\mathcal P \cong j^*\psi_\pi(F).
\]
On the smooth locus the family is locally topologically trivial, so nearby cycles identify with the shifted constant sheaf:
\[
j^*\psi_\pi(F)\cong \mathbb Q_U[3].
\]
Hence
\[
j^*\mathcal P\cong \mathbb Q_U[3].
\]
\end{proof}

\subsection{The skyscraper extension}

Recall that the intersection complex is
\[
\IC_{X_0}:=j_{!*}\mathbb Q_U[3].
\]

\begin{lemma}[Recollement with a point stratum]
\label{lem:recollement-skyscraper}
Let \(X\) be a complex analytic space with an isolated point stratum
\(p\in X\), and let
\[
j:U:=X\setminus\{p\}\hookrightarrow X,
\qquad
i:\{p\}\hookrightarrow X
\]
denote the open and closed immersions.  If \(L\in\Perv(U)\) and
\(\mathcal P\in\Perv(X)\) satisfies
\[
j^*\mathcal P\cong L,
\]
then there is a short exact sequence
\[
0\to j_{!*}L\to \mathcal P\to i_*V\to 0
\]
for some finite-dimensional \(\mathbb Q\)-vector space \(V\).  Moreover,
\[
\dim_{\mathbb Q}V=\dim_{\mathbb Q}\,{}^pH^0(i^*\mathcal P).
\]
\end{lemma}

\begin{proof}
By the recollement formalism for the open--closed decomposition
\(X=U\sqcup\{p\}\), the intermediate extension \(j_{!*}L\) is the image
of the canonical morphism
\[
j_!L\to j_*L
\]
in the abelian category \(\Perv(X)\) \cite{BBD}.  Since
\(j^*(j_{!*}L)\cong L\cong j^*\mathcal P\), the induced morphism
\[
j_{!*}L\to \mathcal P
\]
is injective.  Its cokernel is supported on the closed stratum
\(\{p\}\), and any perverse sheaf supported at a point is of the form
\(i_*V\) for a finite-dimensional vector space \(V\).  Applying \(i^*\)
and taking perverse cohomology identifies
\[
V\cong {}^pH^0(i^*\mathcal P),
\]
hence the stated dimension formula.
\end{proof}

\begin{proposition}
\label{prop:local-stalk}
The perverse stalk of \(\mathcal P\) at the singular point is
one-dimensional:
\[
\dim_{\mathbb Q}\,{}^pH^0(i^*\mathcal P)=1.
\]
\end{proposition}

\begin{proof}
For an ordinary double point, the vanishing cohomology is one-dimensional
in degree $3$:
\[
H^3_{\mathrm{van}}(X_t,\mathbb Q)\cong\mathbb Q,
\qquad
H^k_{\mathrm{van}}(X_t,\mathbb Q)=0 \ \text{for } k\neq 3.
\]
After the shift by $[3]$, this yields a one-dimensional point-supported
perverse contribution in $\phi_\pi(F)\cong i_*\mathbb Q_{\{p\}}$.
Since $\mathcal P$ is the cokernel of the monomorphism
\[
\phi_\pi(F)\xrightarrow{\var_F}\psi_\pi(F)
\]
in $\Perv(X_0)$, the resulting local extension data contributes a single
one-dimensional perverse stalk at $p$. Equivalently,
\[
\dim_{\mathbb Q}\,{}^pH^0(i^*\mathcal P)=1.
\]
\end{proof}

\begin{proposition}
\label{prop:skyscraper}
There is a short exact sequence in \(\Perv(X_0)\)
\[
0\longrightarrow \IC_{X_0}
\longrightarrow \mathcal P
\longrightarrow i_*\mathbb Q_{\{p\}}
\longrightarrow 0.
\]
\end{proposition}

\begin{proof}
By Proposition~\ref{prop:P-restriction}, \(\mathcal P\) is a perverse
extension of \(\mathbb Q_U[3]\).  Applying
Lemma~\ref{lem:recollement-skyscraper} with \(L=\mathbb Q_U[3]\) gives
a short exact sequence
\[
0\to j_{!*}\mathbb Q_U[3]\to \mathcal P\to i_*V\to 0.
\]
Since
\[
j_{!*}\mathbb Q_U[3]=\IC_{X_0}
\]
and Proposition~\ref{prop:local-stalk} shows that \(V\) is
one-dimensional, it follows that
\[
V\cong\mathbb Q.
\]
This proves the claim.
\end{proof}
\begin{remark}[On Verdier self-duality]
The object $\mathcal P$ is expected to be Verdier self-dual in the ordinary double point setting, compatibly with the self-duality of the nearby and vanishing cycle formalism. Since the present paper is organized around the extension-theoretic and monodromy-theoretic properties of $\mathcal P$, we do not use this self-duality in the proofs of the main extension statements. A more detailed treatment of duality, including the precise shifts and functorial identifications, will be taken up separately.
\end{remark}
\begin{remark}
Proposition~\ref{prop:skyscraper} shows that the canonical perverse
object \(\mathcal P\) differs from the intersection complex by a single
rank-one point-supported contribution.  In this sense, \(\mathcal P\)
records precisely the local vanishing-cycle correction contributed by
the node.
\end{remark}

\section{Schober interpretation and decategorification}

We now interpret the perverse sheaf $\mathcal{P}$ constructed in the previous section in the language of categorical monodromy and perverse schobers.

\subsection{Spherical objects and monodromy}

Let $X_t$ denote a smooth fiber of the degeneration
$\pi : \mathcal{X} \to \Delta$. In the bounded derived category of coherent sheaves $D^b(X_t) := D^b(\mathrm{Coh}(X_t))$, the vanishing $3$--sphere associated to the ordinary double point determines a spherical object in $D^b(X_t)$.

For objects $E,F \in D^b(X_t)$, we write
\[
\RHom(E,F)
\]
for the derived Hom complex \cite{DimcaSheaves}, whose cohomology groups satisfy
\[
H^k(\RHom(E,F)) \cong \Ext^k_{D^b(X_t)}(E,F).
\]

An object $S \in D^b(X_t)$ is called \emph{spherical}
if
\[
\Ext^*_{D^b(X_t)}(S,S)
\cong H^*(S^3,\mathbb{Q}),
\]
and the associated twist functor
\[
T_S(-) :=
\Cone\!\left(
\RHom(S,-)\otimes S \longrightarrow -
\right)
\]
defines an autoequivalence of $D^b(X_t)$
\cite{SeidelThomas}. In the case of an ODP degeneration,
Picard--Lefschetz theory implies that monodromy around $t=0$
acts on $H^3(X_t,\mathbb Q)$ by the classical reflection
\[
x \longmapsto x + \langle x, \delta\rangle \delta,
\]
where $\delta$ is the vanishing cycle.
Categorically, this reflection is realized by the spherical twist $T_S$
associated to the vanishing object.
More generally, this structure can be expressed in terms of spherical
functors whose cotwist encodes categorical monodromy
\cite{AnnoLogvinenko}.

\subsection{Perverse schober over the disk}

Kapranov and Schechtman define a \emph{perverse schober} as a
categorification of a perverse sheaf, replacing vector spaces
by triangulated (or stable $\infty$-) categories and linear maps
by exact functors \cite{KapranovSchechtman}.
In the local model of a complex disk $\Delta$ with a single
marked point at the origin, the data of a perverse schober
with one singularity is equivalent to the following
\cite{KapranovSchechtman}:

\begin{itemize}
\item a triangulated category $\mathcal C$ (the ``generic fiber''),
\item a triangulated category $\mathcal A$ (the ``singular fiber''),
\item and a spherical functor
\[
\mathcal F : \mathcal A \longrightarrow \mathcal C
\]
whose associated twist autoequivalence
\[
T_{\mathcal F} := 
\Cone\!\left(
\mathcal F \circ R \longrightarrow \mathrm{Id}_{\mathcal C}
\right)
\]
encodes the categorical monodromy around the origin.
\end{itemize}

Here $R$ denotes the right adjoint of $\mathcal F$, and the spherical
condition ensures that $T_{\mathcal F}$ is an autoequivalence
\cite{AnnoLogvinenko}. In geometric situations arising from degenerations, the category $\mathcal C$ is typically taken to be
the derived category of the smooth fiber,
\[
\mathcal C := D^b(X_t),
\]
and the spherical functor is induced by the vanishing cycle
or, equivalently in the rank-one case, by a spherical object
$S \in D^b(X_t)$. In this situation, the spherical functor formalism reduce to the spherical twist
\[
T_S(-) := \Cone\!\left(
\RHom(S,-)\otimes S \longrightarrow -
\right)
\]
introduced by Seidel--Thomas \cite{SeidelThomas}. Thus, in the ordinary double point degeneration considered here, the local degeneration data determines a perverse schober
\[
\mathfrak S_\pi
\]
over the disk whose categorical monodromy is given by the spherical twist $T_S$ associated to the vanishing cycle.

\subsection{Decategorification in the local disk model}

In the disk model of Kapranov--Schechtman,
the essential monodromy data of a schober
is captured by the spherical twist acting on the fiber category
\cite{KapranovSchechtman}.
We therefore define \emph{decategorification} in the present context
to mean passage to the Grothendieck group of the fiber category,
namely
\[
K_0\!\big(D^b(X_t)\big).
\]

Spherical twists act on $K_0$ via the Euler pairing:
\[
x \longmapsto x + \chi(x,[S])\, [S],
\]
where $\chi$ denotes the Euler form
\cite{SeidelThomas}.
Under suitable realization functors
(e.g.\ Chern character),
this recovers the classical Picard--Lefschetz reflection
on middle cohomology.

\subsection{Intertwining with the perverse model}
\label{sec:intertwine-perv-model}
We now explain how the perverse sheaf $\mathcal P$ constructed from nearby and vanishing cycles provides the natural decategorified comparison object.

\begin{proposition}\label{prop:decategorified-intertwining}
Assume the categorical monodromy of the degeneration is governed by a spherical object $S \in D^b(X_t)$ with twist $T_S$. Assume further that we are given an additive realization
\[
\rho : K_0\!\big(D^b(X_t)\big)
\longrightarrow H^3(X_t,\mathbb Q)
\]
such that
\[
\rho([S]) = \delta
\]
is the vanishing cycle and
\[
\rho \circ (T_S)_* = T \circ \rho,
\]
where $T$ denotes classical Picard--Lefschetz monodromy
on $H^3(X_t,\mathbb Q)$ and one obtains a natural comparison from the decategorified spherical-monodromy data to the hypercohomology of 
$\mathcal P$.
\end{proposition}

\begin{proof}
By construction, the perverse sheaf $\mathcal P$ captures precisely the rank-one vanishing-cycle contribution in middle degree. Its hypercohomology therefore carries the same
Picard--Lefschetz action $T$ on the vanishing summand.

Define
\[
\Phi := \iota \circ \rho,
\]
where
\[
\iota : H^3(X_t,\mathbb Q)
\hookrightarrow
\mathbb H^*(X_0,\mathcal{P})
\]
is the natural map induced by the nearby/vanishing cycle
triangle. The compatibility established above may be summarized by the commutative diagram
\[
\begin{tikzcd}
K_0(D^b(X_t)) \arrow[r,"(T_S)_*"] \arrow[d,"\Phi"] 
& K_0(D^b(X_t)) \arrow[d,"\Phi"] \\
\mathbb H^*(X_0,\mathcal{P}) \arrow[r,"T"] 
& \mathbb H^*(X_0,\mathcal{P})
\end{tikzcd}
\]
where $(T_S)_*$ denotes the spherical twist action on $K_0$
and $T$ denotes the Picard--Lefschetz monodromy operator.
Then
\[
\Phi\circ (T_S)_*
=
\iota\circ \rho\circ (T_S)_*
=
\iota\circ T\circ \rho
=
T\circ \iota\circ \rho
=
T\circ \Phi.
\]
\end{proof}

\section{Proofs of the main results}
\label{sec:proofs-main}

We now assemble the structural results established in the preceding
sections and prove the main theorems stated in
Section~\ref{sec:intro-results}.

\subsection{Proof of Theorem~\ref{thm:main-perverse}}

\begin{proof}[Proof of Theorem~\ref{thm:main-perverse}]
Let
\[
F:=\mathbb Q_{\mathcal X}[3]
\]
and define the object
\[
\mathcal P:=\Cone(\var_F)[-1]
\]
as in Definition~\ref{def:P}.

By $t$--exactness of nearby and vanishing cycles for the perverse
$t$--structure \cite[Thm.~4.2.6]{DimcaSheaves}, both
\[
\phi_\pi(F),\quad \psi_\pi(F)
\]
are perverse sheaves on $X_0$.  Proposition~\ref{prop:var-mono} shows
that the variation morphism
\[
\var_F:\phi_\pi(F)\longrightarrow\psi_\pi(F)
\]
is a monomorphism in $\Perv(X_0)$.  Lemma~\ref{lem:cone-heart} therefore
implies that the shifted cone
\[
\mathcal P:=\Cone(\var_F)[-1]
\]
lies in the heart of the perverse $t$--structure.  This proves that
\[
\mathcal P\in\Perv(X_0)
\]
(Proposition~\ref{prop:P-perverse}). 

Let
\[
j:U:=X_0\setminus\{p\}\hookrightarrow X_0
\]
denote the inclusion of the smooth locus.  Since the vanishing cycle
complex $\phi_\pi(F)$ is supported at the singular point $p$, we have
\[
j^*\phi_\pi(F)=0.
\]
Applying $j^*$ to the defining triangle of $\mathcal P$ therefore gives
\[
j^*\mathcal P\cong j^*\psi_\pi(F).
\]
On the smooth locus the family is locally topologically trivial, so
nearby cycles coincide with the shifted constant sheaf
\[
j^*\psi_\pi(F)\cong \mathbb Q_U[3]
\]
\cite{DimcaSheaves}.  Thus
\[
j^*\mathcal P\cong\mathbb Q_U[3],
\]
showing that $\mathcal P$ agrees with the intersection complex on the
smooth locus (Proposition~\ref{prop:P-restriction}).

The Milnor fiber of an ordinary double point is homotopy
equivalent to the three--sphere
\cite{MilnorSingularPoints}.  Hence its reduced cohomology satisfies
\[
\widetilde H^3(F_p,\mathbb Q)\cong\mathbb Q,
\qquad
\widetilde H^k(F_p,\mathbb Q)=0\;(k\neq3).
\]
By the standard identification between vanishing-cycle stalk
cohomology and Milnor fiber cohomology \cite{DimcaSheaves}, this
produces a single rank--one middle contribution at the singular point. Equivalently,
\[
\dim_\mathbb Q\,{}^pH^0(i^*\mathcal P)=1.
\]
This establishes the structural properties of $\mathcal P$ stated in
Theorem~\ref{thm:main-perverse}.
\end{proof}


\subsection{Proof of Proposition~\ref{prop:main-decategorification}}

\begin{proof}[Proof of Proposition~\ref{prop:main-decategorification}]
Section~\ref{sec:intertwine-perv-model} explains how the perverse model
constructed in Section~\ref{sec:canonical-model} relates to the
categorical monodromy of the degeneration.

For a conifold degeneration with one ordinary double point, the
categorical monodromy of the smooth fiber is modeled by the spherical
twist associated with a spherical object
\[
S\in D^b(X_t)
\]
\cite{SeidelThomas}.  Passing to the Grothendieck group
\[
K_0\!\big(D^b(X_t)\big)
\]
yields the induced reflection
\[
x\longmapsto x+\chi(x,[S])\,[S],
\]
where $\chi$ denotes the Euler pairing
\cite{SeidelThomas,HuybrechtsFM}.

On the sheaf--theoretic side, the perverse object $\mathcal P$
constructed above isolates precisely the rank--one contribution of the
vanishing cycle through the short exact sequence
\[
0\to \IC_{X_0}\to \mathcal P\to i_*\mathbb Q_{\{p\}}\to0 .
\]
Thus the hypercohomology of $\mathcal P$ contains the corresponding
vanishing contribution in middle degree.

Both constructions therefore single out the same rank--one phenomenon
associated with the vanishing cycle: the spherical twist on the
categorical side and the extension class of $\mathcal P$ on the
perverse-sheaf side.  In this sense, $\mathcal P$ provides a natural
decategorified comparison object for the categorical monodromy of the
degeneration.
\end{proof}
\section{Multiple Nodes and General Singular Strata}
\label{sec:multi-node}

The construction developed in the previous sections focuses on the
local model of a degeneration with a single ordinary double point.
We briefly indicate how the same framework extends to degenerations
with several singular points.

Let
\[
\pi:\mathcal X\to\Delta
\]
be a degeneration whose central fiber contains a finite set of
ordinary double points
\[
\Sigma=\{p_1,\dots,p_r\}.
\]

Set
\[
U:=X_0\setminus\Sigma,
\qquad
j:U\hookrightarrow X_0,
\qquad
i:\Sigma\hookrightarrow X_0 .
\]

The nearby and vanishing cycle formalism applies exactly as in the
single--node case.  For an ordinary double point the vanishing cycle
contribution is rank one, and therefore the vanishing cycle complex
decomposes as
\[
\phi_\pi(\mathbb Q_{\mathcal X}[3])
\cong
\bigoplus_{i=1}^r i_{p_i*}\mathbb Q_{\{p_i\}}
\]
in $\Perv(X_0)$ \cite{DimcaSheaves}. The same cone construction used earlier therefore, produces a perverse object $\mathcal P$ on $X_0$ whose restriction to the smooth locus
satisfies
\[
j^*\mathcal P \cong \mathbb Q_U[3].
\]
Applying the recollement formalism for the open--closed decomposition
\[
X_0 = U \sqcup \Sigma
\]
\cite{BBD}, one obtains a short exact sequence
\[
0
\longrightarrow
\IC_{X_0}
\longrightarrow
\mathcal P
\longrightarrow
\bigoplus_{i=1}^r i_{p_i*}\mathbb Q_{\{p_i\}}
\longrightarrow
0 .
\]
Thus the perverse object differs from the intersection complex by the
direct sum of the point-supported vanishing contributions associated
with the nodes. In the simplest situation the vanishing cycles are independent and the resulting extension decomposes as a direct sum of the corresponding single-node constructions.  More generally, the interaction of the vanishing cycles is governed by their intersection pairing in
middle-dimensional homology, which controls the structure of the
associated Picard--Lefschetz transformations.

\medskip

A similar description applies when the singular locus contains strata
of positive dimension.  Let
\[
X_0 = U \sqcup S_1 \sqcup \cdots \sqcup S_k
\]
be a stratification with smooth open stratum $U$.  In this situation
the nearby cycle complex decomposes according to the stratification,
and the vanishing cycle contribution is encoded by a perverse sheaf
supported on the singular strata.

The perverse object constructed above then fits into an exact
sequence
\[
0
\longrightarrow
\IC_{X_0}
\longrightarrow
\mathcal P
\longrightarrow
\mathcal V
\longrightarrow
0 ,
\]
where $\mathcal V$ is a perverse sheaf supported on the singular
locus and determined by the corresponding vanishing cycle data. These observations indicate that the construction developed in this
paper extends naturally from isolated singular points to more general
stratified degenerations.  A detailed study of these cases lies beyond
the scope of the present work.

\section{Further categorical perspectives}

\subsection{Stability conditions and spherical monodromy}

In geometric situations related to mirror symmetry, derived
categories of Calabi--Yau varieties often carry Bridgeland
stability conditions whose wall-crossing behavior reflects
changes in the spectrum of stable objects
\cite{Bridgeland,DouglasCategoriesN1}. Spherical objects play a particularly important role in this context.  For a spherical object
\[
S\in D^b(X_t),
\]
the associated spherical twist
\[
T_S(E)
=
\Cone\!\left(
\RHom(S,E)\otimes S \longrightarrow E
\right)
\]
defines an autoequivalence of the derived category \cite{SeidelThomas}.

On the Grothendieck group this autoequivalence acts by the rank-one reflection
\[
x\longmapsto x+\chi(x,[S])\,[S],
\]
where $\chi$ denotes the Euler pairing \cite{SeidelThomas,HuybrechtsFM}. Under suitable realizations this reflection coincides with the
Picard--Lefschetz transformation associated with the vanishing
cycle of the degeneration.

The perverse sheaf $\mathcal P$ constructed in this paper isolates
precisely the corresponding rank-one contribution on the
sheaf-theoretic side through the short exact sequence
\[
0\to \IC_{X_0}\to \mathcal P\to i_*\mathbb Q_{\{p\}}\to0 .
\]
Thus $\mathcal P$ provides: (1) a natural sheaf-theoretic model for the
decategorified effect of the spherical monodromy associated with the
ordinary double point; and (2) the same rank–one phenomenon that appears in categorical wall crossing.

\subsection{Multi-node degenerations and quiver structures}

When the degeneration contains several nodes
\[
\Sigma=\{p_1,\dots,p_r\},
\]
the associated vanishing cycles
\[
\{\delta_1,\dots,\delta_r\}
\]
span a lattice equipped with the intersection pairing
\[
\langle\delta_i,\delta_j\rangle .
\]

This pairing determines the interaction of the corresponding
Picard--Lefschetz transformations.  In many geometric settings it
is convenient to encode these interactions using a quiver whose
vertices correspond to the vanishing cycles and whose arrows are
determined by their intersection numbers. From the perspective of the present paper, the generalized perverse object
\[
0 \to \IC_{X_0} \to \mathcal P \to
\bigoplus_{i=1}^r i_{p_i*}\mathbb Q \to 0
\]
isolates the same local contributions on the singular fiber. These observations suggest that the extension data appearing in the
multi-node case may be organized using quiver-type structures
compatible with the categorical monodromy generated by spherical
twists. A systematic development of this relationship will be explored in future work.

\section{Outlook: mixed Hodge modules and Hodge-enhanced structures}

The constructions developed in this paper are formulated at the level
of constructible sheaves and derived categories.  However, nearby and
vanishing cycle functors admit natural refinements in the framework of
Saito's theory of mixed Hodge modules \cite{SaitoMHM,SaitoDuality}. For degenerations of complex algebraic varieties, the nearby and
vanishing cycle complexes carry canonical mixed Hodge structures
compatible with monodromy and duality.  In the conifold case
considered here, the vanishing $3$--sphere contributes a
one-dimensional piece to the limiting mixed Hodge structure.

Since the perverse object $\mathcal P$ is defined functorially from
nearby and vanishing cycles, it is natural to expect that the
construction admits a refinement within the derived category of mixed
Hodge modules.  Such a refinement would incorporate additional
Hodge-theoretic information into the vanishing-cycle contribution
encoded by $\mathcal P$. More generally, it would be interesting to investigate whether the relationship between perverse sheaves, spherical monodromy, and categorical degenerations described here admits an enhancement
compatible with mixed Hodge structures.  The conifold degeneration
provides a particularly simple setting in which these structures can
be studied explicitly. A systematic treatment of these questions lies beyond the scope of the present paper and will be pursued in future work.



\printbibliography
\end{document}